 \RequirePackage[table]{xcolor}

\documentclass[11pt]{amsart}

\usepackage[colorlinks,citecolor=blue,urlcolor=blue,linkcolor=blue]{hyperref}
\usepackage[top=1 in,bottom=1 in, left=1 in, right= 1 in, includehead,includefoot]{geometry}
\hypersetup{colorlinks = true,citecolor=blue,urlcolor=blue,linkcolor=blue,pdfpagemode = UseNone}

\usepackage{enumerate}

  \usepackage{tikz}
 \usetikzlibrary{patterns}
\usepackage[sort&compress]{natbib}
\usepackage{amsmath,amsthm,amsfonts,amssymb}
 \usepackage{dsfont}

 \usepackage{graphicx,fancybox,pgf,amsmath,float,amssymb,framed,color,pgf}

\makeatletter
\renewcommand*\env@matrix[1][\arraystretch]{%
  \edef\arraystretch{#1}%
  \hskip -\arraycolsep
  \let\@ifnextchar\new@ifnextchar
  \array{*\c@MaxMatrixCols c}}
\makeatother

 \usepackage{framed,pgf}
\newcounter{oldeq}
\newcounter{usesofarxiv}
 \newcommand{\arxiv}[1]{
\setcounter{oldeq}{\value{equation}}
 \addtocounter{usesofarxiv}{1}
 \setcounter{equation}{0}
\def\theoldeq{\theequation}
\def\theequation{x-\arabic{usesofarxiv}.\arabic{equation}}
\def\theequation{\arabic{section}.\arabic{usesofarxiv}.\arabic{equation}}
\def\theequation{\thesection.\arabic{usesofarxiv}.\arabic{equation}}
  \colorlet{shadecolor}{gray!10}
{
\begin{shaded} 
\footnotesize
#1  \normalsize  
\end{shaded}
   \setcounter{equation}{\value{oldeq}}
\numberwithin{equation}{section}
}}

\newcommand{\comment}[1]{}
\newcommand{\commentYZ}[1]{}
\newcommand{\longcomment}[1]{}
\renewcommand{\longcomment}[1]{\ovalbox{\begin{minipage}{.9\textwidth}\color{blue}#1\end{minipage}}}
\renewcommand{\comment}[1]{{\color{magenta}\ovalbox{#1}}}

\newcommand{\ind}{{\bf 1}}

\def\inddd#1{{\ind}_{\left\{#1\right\}}}

\newcommand{\proba}{\mathbb P}

\newcommand{\inv}{^{-1}}

\newcommand{\eqnh}{\begin{eqnarray*}}
\newcommand{\eqne}{\end{eqnarray*}}
\newcommand{\eqnhn}{\begin{eqnarray}}
\newcommand{\eqnen}{\end{eqnarray}}
\newcommand{\equh}{\begin{equation}}
\newcommand{\eque}{\end{equation}}

\def\summ#1#2#3{\sum_{#1 = #2}^{#3}}
\def\prodd#1#2#3{\prod_{#1 = #2}^{#3}}
\def\sif#1#2{\sum_{#1=#2}^\infty}

\def\topp#1{^{(#1)}}

\def\abs#1{\left|#1\right|}

\def\ccbb#1{\left\{#1\right\}}
\def\sccbb#1{\{#1\}}
\def\pp#1{\left(#1\right)}

\def\bb#1{\left[#1\right]}

\def\mmid{\;\middle\vert\;}

\def\floor#1{\left\lfloor #1 \right\rfloor}

\def\vv#1{{\boldsymbol #1}}
\def\vvalpha{{\vv\alpha}}
\def\vvbeta{{\vv\beta}}

\def\qmand{\quad\mbox{ and }\quad}

\def\qmwith{\quad\mbox{ with }\quad}
\def\mfa{\mbox{ for all }}

\def\wt#1{\widetilde{#1}}

\def\weakto{\Rightarrow}

\def\Z{{\mathbb Z}}

\renewcommand{\ind}{{\bf 1}}

\def\inddd#1{{\ind}_{\left\{#1\right\}}}

\newcommand{\hidecomment}[1]{}
\theoremstyle{plain}
\newtheorem{theorem}{Theorem}[section]
\newtheorem{corollary}[theorem]{Corollary}
\newtheorem{lemma}[theorem]{Lemma}
\newtheorem{proposition}[theorem]{Proposition}

\theoremstyle{remark}
\newtheorem{remark}[theorem]{Remark}

\newcommand{\calM}{\mathcal{M}}

\def\<{\langle}
\def\>{\rangle}

\newcommand{\RR}{\mathds{R}}
\newcommand{\CC}{\mathds{C}}
\newcommand{\NN}{\mathds{N}}
\newcommand{\ZZ}{\mathds{Z}}

\newcommand{\eps}{\varepsilon}
\newcommand{\la}{\lambda}

\newcommand{\EE}{\mathds{E}}
\renewcommand{\Pr}{\mathds{P}}

\def\<{\langle}
\def\>{\rangle}

\newcommand{\mM}{{\vv M}}
\newcommand{\mI}{{\vv I}}

\newcommand{\mB}{{\vv B}}

\numberwithin{equation}{section}

\title[Random Motzkin paths near boundary]{Limit theorems for random Motzkin paths near boundary}

\author{W\l odzimierz Bryc}
\address
{
W\l odzimierz Bryc\\
Department of Mathematical Sciences\\
University of Cincinnati\\
2815 Commons Way\\
Cincinnati, OH, 45221-0025, USA.
}
\email{wlodzimierz.bryc@uc.edu}
\author{Yizao Wang}
\address
{
Yizao Wang\\
Department of Mathematical Sciences\\
University of Cincinnati\\
2815 Commons Way\\
Cincinnati, OH, 45221-0025, USA.
}
\email{yizao.wang@uc.edu}

\begin{document}\sloppy

\begin{abstract}
We consider Motzkin paths of length $L$, not fixed at zero at both end points, with constant weights on the edges and general weights on the end points. We investigate, as the length $L$ tends to infinity,  the limit behaviors of (a) boundary measures induced by the weights on both end points and (b) the segments of the sampled Motzkin path viewed as a process starting from each of the two end points, referred to as boundary processes.
 Our first result concerns the case when the induced boundary measures have finite first moments. Our second result concerns when the boundary measure on the right end point is a generalized geometric measure with parameter $\rho_1\ge 1$, so that this is an infinite measure and yet it induces a probability measure for random Motzkin path when $\rho_1$ is
 not too large. The two cases under investigation reveal a phase transition. In particular, we show that the limit left boundary processes in the two cases
 have the same transition probabilities as random walks conditioned to stay non-negative.

\end{abstract}
\maketitle

\arxiv{This is an expanded  version of the paper,  with corrected typos and additional clarifications. }
\section{Introduction}

\subsection{Weighted Motzkin paths without fixed boundary points}%

 A Motzkin path of length $L\in \ZZ_{\geq 1}$ is a sequence
of lattice points $(\vv x_0,\dots, \vv x_L)$ in  $\ZZ_{\geq 0}\times  \ZZ_{\geq 0}$, such that $\vv x_j = (j, n_j)$ with $ |n_{j}-n_{j-1}| \leq  1$. An edge $(\vv x_{j-1},\vv x_j)$  is called an up step  if
$n_{j}-n_{j-1}=1$, a down step  if $n_{j}-n_{j-1}=-1$ and a horizontal step if $n_{j}-n_{j-1}=0$.
Each such path can be identified with a  sequence of non-negative integers that specify the starting point $(0,n_0)$ and consecutive   values $n_j$  along the vertical axis at step $j\geq 1$. We shall write  $\vv\gamma=(\gamma_0,\gamma_1,\dots,\gamma_L)$ with $\gamma_j = n_j$  for such a sequence and refer to $\vv\gamma$ as a Motzkin path.   By $\calM_{i,j}\topp L$ we denote the family of all Motzkin paths of length $L$ with the initial altitude $\gamma_0=i$ and the final altitude $\gamma_L=j$. We also refer to $\gamma_0$ and $\gamma_L$ as the boundary/end points of the path. Here, we follow the standard terminology; see  \citet[Definition V.4, page 319]{flajolet09analytic}
 or \cite{Viennot-1984a}.

To introduce random Motzkin paths, we start by assigning  weights to the edges and to the end-points of a Motzkin path. In general,  weights for the edges arise from three sequences $\vv a=(a_j)_{j\geq 0}$, $\vv b=(b_j)_{j\geq 0}$, $\vv c=(c_j)_{j\geq 0}$ of positive real numbers.
For a path $\vv\gamma=(\gamma_0=i,\gamma_1,\dots,\gamma_{L-1},\gamma_L=j)\in\calM_{i,j}\topp L$ we define its (edge) weight
\begin{equation}\label{w}
  w(\vv\gamma)=\prod_{k=1}^L a_{\gamma_{k-1}}^{\eps_k^+} b_{\gamma_{k-1}}^{\eps_k^0}c_{\gamma_{k-1}}^{\eps_k^-},
\end{equation}
where %
\[%
   \eps_k^+(\vv\gamma):=\ind_{\gamma_{k}>\gamma_{k-1}}, \quad\; \eps_k^-(\vv\gamma):=\ind_{\gamma_{k}<\gamma_{k-1}}, \quad \; \eps_k^0(\vv\gamma):=\ind_{\gamma_{k}=\gamma_{k-1}}, \quad k=1,\dots, L.
\]%
That is, the edge weight is multiplicative in the edges, we take $\vv a$, $\vv b$ and $\vv c$ as the weights of the up steps, horizontal steps and down steps, and the weight of a step depends  on the altitude of the left-end of an edge.
Note that the value of $c_0$ does not contribute to $w(\vv\gamma)$.
Since $\calM_{i,j}\topp L$ is a finite set, the normalization constants
\begin{equation}\label{Wij}
  \mathfrak W
  \topp L_{m,n}=\sum_{\vv\gamma\in\calM_{m,n}\topp L} w(\vv\gamma)
\end{equation}
are well defined for all $m,n\geq 0$.

  \begin{figure}[t]
  \begin{tikzpicture}[scale=1]

 \draw[->] (0,0) to (0,3);

 \draw[->] (0,0) to (11,0);
\draw[-,thick] (0,2) to (1,1);
\draw [fill] (0,2) circle [radius=0.05];
 \node[above] at (.6,1.6) {\footnotesize  $c_2$};
\draw[-,thick] (1,1) to (2,1);
\draw [fill] (1,1) circle [radius=0.05];
 \node[above] at (1.6,1) {\footnotesize  $b_1$};
\draw[-,thick] (2,1) to (3,0);
\draw [fill] (2,1) circle [radius=0.05];
 \node[above] at (2.6,0.6) {\footnotesize  $c_1$};
\draw[-,thick] (3,0) to (4,1);
\draw [fill] (3,0) circle [radius=0.05];
 \node[above] at (3.4,0.6) {\footnotesize  $a_0$};
\draw[-,thick] (4,1) to (5,1);
\draw [fill] (4,1) circle [radius=0.05];
 \node[above] at (4.5,1) {\footnotesize  $b_1$};
\draw[-,thick] (5,1) to (6,2);
\draw [fill] (5,1) circle [radius=0.05];
 \node[above] at (5.5,1.6) {\footnotesize  $a_1$};
\draw[-,thick] (6,2) to (7,1);
\draw [fill] (6,2) circle [radius=0.05];
 \node[above] at (6.5,1.6) {\footnotesize  $c_2$};
\draw[-,thick] (7,1) to (8,0);
\draw [fill] (7,1) circle [radius=0.05];
 \node[above] at (7.5,0.6) {\footnotesize  $c_1$};
\draw[-,thick] (8,0) to (9,1);
\draw [fill] (8,0) circle [radius=0.05];
 \node[above] at (8.4,0.6) {\footnotesize  $a_0$};
 \draw [fill] (9,1) circle [radius=0.05];

   \node[below] at (1,0) {  $1$};
      \node[below] at (2,0) {  $2$};
       \node[below] at (3,0) {  $3$};
        \node[below] at (4,0) {  $4$};
         \node[below] at (5,0) {  $5$};
          \node[below] at (6,0) {  $6$};
           \node[below] at (7,0) {  $7$};
            \node[below] at (8,0) {  $8$};
             \node[below] at (9,0) {  $9$};
\end{tikzpicture}
\caption{\label{Fig1}    Motzkin path $\vv\gamma=(2,1,1,0,1,1,2,1,0,1)\in\calM\topp 9$     with weight contributions marked at the edges. The probability of selecting the path shown above from $\calM\topp 9$ is
$\Pr_9(\vv\gamma)=\frac{\alpha_2 \beta_1}{\mathfrak C_{\vv \alpha,\vv \beta, 9} }  b_1^2 a_0^2 a_1 c_1^2  c_2^2$.
}
\end{figure}
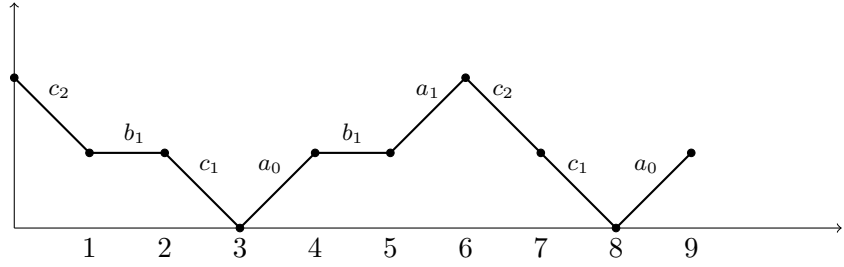

In addition to the weights of the edges, we also weight the two end points.
To this end we choose two additional non-negative sequences $\vv \alpha=(\alpha_m)_{m\ge 0}$ and $\vv \beta=(\beta_n)_{n\ge 0}$ such that
\begin{equation}
   \label{CL} \mathfrak{C}_{\vv\alpha,\vv\beta,L}:=\sum_{m,n\geq 0} \alpha_m \mathfrak W\topp L_{m,n}\beta_n\in(0,\infty).
\end{equation}
In our main results we shall assume %
\equh\label{eq:abc}
\vv a=(1,1,\dots), \quad \vv b=(\sigma,\sigma,\dots), \quad \vv c =(1,1,\dots).
\eque
Then,
normalization constants
$\{\mathfrak W\topp L_{m,n}\}_{m,n\geq 0}$ are also bounded, as  $\mathfrak W\topp L_{m,n}=0$ for $|m-n|>L$. So %
$0<\mathfrak C_{\vv\alpha,\vv\beta,L}<\infty$  if for example  $\sif n0 \alpha_n<\infty$, $\sif n0 \beta_n<\infty$ and $\alpha_n\beta_m>0$ for some $|m-n|\leq L$.

With finite normalizing constant \eqref{CL}, the countable set  $\calM\topp L=\bigcup_{m,n\geq 0} \calM_{m,n}\topp L$ becomes a probability space with the discrete probability measure
\[%
\Pr_L(\vv\gamma)\equiv  \Pr_{\vvalpha,\vvbeta,L}(\{\vv\gamma\})=\frac{\alpha_{\gamma_0}\beta_{\gamma_L}}{\mathfrak{C}_{\vvalpha,\vvbeta,L}} w(\vv\gamma), \quad \vv\gamma\in\calM\topp L.
\]%
For an illustration, see Figure \ref{Fig1}.
We now consider a random vector $(\gamma_0\topp L,\dots,\gamma_L\topp L)\in\calM\topp L$ sampled from $\Pr_L$, and extend it to an infinite process denoted by
\[
\vv\gamma\topp L:= \ccbb{\gamma_k\topp L}_{k\ge 0}, \qmwith \gamma_k\topp L := 0 \mbox{ if } k>L.
\]
Similarly, we also introduce the infinite reversed process by %
\[
\wt{\vv\gamma}\topp L:=\ccbb{\wt\gamma_k\topp L}_{k\ge 0} \qmwith \wt\gamma_k\topp L :=
\begin{cases}\gamma_{L-k}\topp L & \mbox{ if } k=0,\dots,L\\0 & \mbox{ if } k>L.\end{cases}
\]
Note that  two sequences $\vv\alpha,\vv\beta$ induce    positive, possibly $\sigma$-finite, measures on $\ZZ_{\ge 0}$,
and we refer to both as the {\em boundary measures}. We are interested in the limit of the boundary measures as $L\to\infty$. It turns out that this can be done along with computing the limit processes for both $\vv\gamma\topp L$ and $\wt{\vv\gamma}\topp L$. We refer to the two processes as %
the
(left and right)
{\em boundary processes}.

Our contributions are limit theorems for the boundary processes $\vv\gamma\topp L$ and $\wt{\vv\gamma}\topp L$. Our limit theorems concern two cases. The first is when the boundary measures are both finite and have finite first moments. In the second one, we
consider
{\em general} geometric weights at the right end point with parameter $\rho_1$ which we allow to be larger than one.
 So in this latter case the right boundary measure may become  infinite and  %
 we reveal a phase transition when parameter %
 $\rho_1$ crosses 1. In the limit, the left boundary processes have the same transition probabilities as the  random walks conditioned to stay non-negative \citep{bertoin1994conditioning}.

\subsection{First result: when the boundary measures have finite first moments}
Consider the following transition probabilities for a Markov chain on $\ZZ_{\ge 0}$  from $n$ to $m$:
\begin{equation}\label{M-trans}
\mathsf Q_{n,m}:=\begin{cases}
  \displaystyle
\frac{1}{2+\sigma}  \frac{n+2}{n+1} & \mbox{ if } m=n+1\\\\
  \displaystyle    \frac{\sigma}{2+\sigma}& \mbox{ if } m=n \\\\
  \displaystyle      \frac{1}{2+\sigma}\frac{n}{n+1} & \mbox{ if } m=n-1\ge 0 \\\\
  \displaystyle     0 & \mbox{otherwise}
.  \end{cases}%
\end{equation}
With $\sigma=0$, this corresponds to the discrete 3-dimensional Bessel process introduced in \citet{pitman1975one}.
As one of the reviewers pointed out, in general these are transition probabilities of a $\sigma$-lazy random walk started at $0$ and conditioned on staying positive.
\begin{theorem}
  \label{Thm-uni}
   Assume \eqref{eq:abc} with $\sigma>0$ and
   \equh\label{eq:integrable}
   \sif n0n\alpha_n<\infty \qmand \sif n0 n\beta_n<\infty.
   \eque  Then %
   \[
 \pp{\ccbb{\gamma_k\topp L}_{k\ge 0},\ccbb{\wt\gamma_k\topp L}_{k\ge 0}}\weakto \pp{\ccbb{X_k}_{k\ge 0},\ccbb{X_k'}_{k\ge 0}}
 \quad \mbox{ as $L\to\infty$},
\]
 where the right-hand side are two independent Markov chains with %
 the same %
  transition probabilities $\{\mathsf Q_{n,m}\}_{n,m\ge 0}$ in \eqref{M-trans}, and %
  different
   initial laws of $X_0$ and $X'_0$ respectively as %
\begin{equation}\label{ini-law}
  \Pr(X_0=n)=\frac{1}{C_{\vv\alpha}}(n+1)\alpha_n, \quad \Pr(X'_0=n)=\frac{1}{C_{\vv\beta}}(n+1)\beta_n, \quad n\ge 0,
\end{equation}
with the normalization constants
 \[%
    C_{\vv\alpha}=\sum_{m=0}^\infty(m+1)\alpha_m ,\quad C_{\vv\beta}=\sum_{m=0}^\infty(m+1)\beta_m.
\]%
\end{theorem}
\begin{remark}
We assume $\sigma>0$.
We note that when $\sigma=0$, even for the marginal law of the end points of $\gamma_0\topp L,\gamma_L\topp L$ the asymptotic independence does not hold.
For example, with $\alpha_n=\beta_n=0$ for $n\ge 2$ and  $L=2N$, the end-points are dependent, $\gamma_0=\gamma_{2N}\in\{0,1\}$, and the limiting law is different and %
depends on both sequences $\vv \alpha, \vv \beta$.
$$\lim_{N\to\infty}\Pr(\gamma_{2N}\topp{2N}=1)= \frac{4 \alpha_1 \beta_1 }{\alpha_0\beta_0+4\alpha_1\beta_1} . $$
This implies that Theorem \ref{Thm-uni} does not hold for $\sigma=0$.
\end{remark}

\begin{remark}
Note that we do not scale in neither time nor the altitudes, and our limit theorems should be considered as {\em microscopic}  {and {\em local}}: both limit processes, despite that they are of infinite length, characterize asymptotic behaviors near the
end-points of a random Motzkin path as  the distance between %
the end-points goes to infinity. {Our results should be compared to limit theorems for random walks conditioned to stay non-negative, which concern {\em initial behaviors of the random walk} \citep{bertoin1994conditioning}.}

Another type of limit theorems concern {\em macroscopic} limit, with scaling both in time and altitudes. See for example \cite{kaigh1976invariance} and \cite{Bryc-Wang-2019}.
\end{remark}

\subsection{Second result: phase transition with general geometric boundary measures}
The assumption \eqref{eq:integrable}  says that the normalized boundary measures on both end points have finite mean. Our next result concerns the situation when the boundary measure on the right end point is an infinite measure. In fact,   we focus on choosing  geometric weights $\vv\beta$ as
\[
 \beta_n=\rho_1^n, \quad n\ge0,
\]
with  $\rho_1\ge 1$.
For $\Pr_L\equiv \Pr_{\vvalpha,\vvbeta,L}$ to be a well-defined probability measure, we need %
\equh\label{eq:rho_1}
0<\sif m0 (m+1)\alpha_m\rho_1^m<\infty %
\eque
to guarantee $0<\mathfrak C_{\vvalpha,\vvbeta,L}<\infty$.

 In this regime,
we obtain a different Markov chain in the limit
when $\rho_1>1$.
Its transition probabilities are
\begin{equation}
  \label{Z-trans}
  \mathsf Q_{n,m}\topp\rho:=\begin{cases}
\displaystyle \frac{1}{\rho+1/\rho+\sigma}  \;  \frac{\rho^{n+2} - 1/\rho^{n+2}}{\rho^{n+1}-1/\rho^{n+1}} & \mbox{ if } m=n+1\\\\
\displaystyle \frac{1}{\rho+1/\rho+\sigma}    \sigma & \mbox{ if } m=n \\\\
\displaystyle \frac{1}{\rho+1/\rho+\sigma}   \; \frac{\rho^n-1/\rho^{n}}{\rho^{n+1}-1/\rho^{n+1}} & \mbox{ if } m=n-1 \\\\
    0 & \mbox{otherwise},
  \end{cases} %
\end{equation}
where $\rho\ne 1$.
{We note that $\rho=1$ is a removable singularity, and in a more compact form one can write
\[
\mathsf Q_{n,m}\topp\rho:=\frac1{\rho^2+1+\rho\sigma}\frac{[m+1]_{\rho^2}}{[n+1]_{\rho^2}}\pp{\inddd{m-n= 1}
+\rho\sigma\inddd{m-n= 0}+\rho^2\inddd{m-n= -1}}, \quad\rho\ge 0,
\]
with $[n]_{\rho^2}:=1+\rho^2+\cdots+\rho^{2n-2}$, and in particular
 $\mathsf Q_{n,m} = \mathsf Q_{n,m}\topp 1$.
}

\begin{remark} %
When $\rho=1$ ($\rho>1$ resp.), $\mathsf Q\topp \rho$ with $\sigma=0$ correspond to the conditional law of a simple random walk (a biased random walk drifting to $\infty$ resp.) staying non-negative, as summarized in \citet{bertoin1994conditioning}. For $\sigma=0$, the Markov process appeared also previously in \citet[(2.a)]{miyazaki1989theorem}.
\end{remark}

\begin{theorem}\label{T-rho}
Consider weighted random Motzkin paths with general weights $\vvalpha$ on the left end points and geometric weights $\beta_n = \rho_1^n, n\ge 0$ on the right end points as above. Suppose $\rho_1\ge 1$ and \eqref{eq:rho_1} holds.
Then $\gamma_L\topp L$ is not tight as $L\to\infty$, and
   \[
   \pp{\ccbb{\gamma_k\topp L}_{k\ge 0},\ccbb{\wt\gamma_k\topp L-\wt\gamma_{k-1}\topp L}_{k\ge 1}} \weakto \pp{\ccbb{Z_k}_{k\ge 0}, \ccbb{ \xi_k}_{k\ge 1}},
   \]
   where on the right-hand side, $\sccbb{Z_k}_{k\ge 0}$ is the Markov chain with transition probabilities $\{\mathsf Q_{n,m}\topp {\rho_1}\}_{n,m\ge 0}$ in \eqref{Z-trans}  and initial law $Z_0$ given by
   \equh\label{eq:Z_0 1}
\Pr( Z_0 = n) = \frac1{\mathfrak C_{\vvalpha,\rho_1}}\alpha_n \frac{\rho_1^{n+1}-\rho_1^{-(n+1)}}{\rho_1-\rho_1\inv},\quad n\ge 0,
\eque
with $\mathfrak C_{\vvalpha,\rho_1} = \sif n0\alpha_n (\rho_1^{n+1}-\rho_1^{-(n+1)})/(\rho_1-\rho_1\inv)$,
the family $\{\xi_k\}_{k\ge1}$ are i.i.d.~random variables with
\[%
     \Pr(\xi_1=\epsilon)=\frac{1}{\rho_1+1/\rho_1+\sigma}\times \begin{cases}
       1/\rho_1 & \mbox{ if } \epsilon=1\\
       \sigma & \mbox{ if } \epsilon=0 \\
       \rho_1 & \mbox{ if } \epsilon=-1,
     \end{cases}
\]%
   and the two families are independent.
   \arxiv{For $\rho_1=1$,  we extend these formulas by continuity:  we use
    $\frac{\rho^{n+1}-\rho^{-(n+1)}}{\rho-\rho\inv}\Big|_{\rho=1}:=n+1$, and we use $\mathsf Q$ from  \eqref{M-trans} instead of
    $\mathsf Q\topp 1$.}
\end{theorem}
\begin{remark}
There are two reasons that we work with the general geometric measure on the right end point. First, it seems that unlike when under first-moment assumption \eqref{eq:integrable} we could have a unified approach, based on a formula by \citet{Viennot-1984a}, when \eqref{eq:integrable} is violated one has to go through model-specific calculations. Second, the choice of geometric boundary measures is motivated by recent advances on asymmetric simple exclusion processes with open boundary. See Section \ref{sec:motivation}. The choice of geometric boundary measures is crucial at two steps. See Remark \ref{rem:geometric}. It would be interesting to work out another example of infinite boundary measures. {It is not clear to us whether other types of limit boundary processes may arise.}
\end{remark}
\begin{remark}%
When $\rho_1>1$, the boundary measure on the right end point  has a non-negligible influence,  depending specifically on $\rho_1$, on both the initial and transitional laws of the left boundary process. When $0\le \rho_1\le 1$, the  law of the left boundary process does not depend on the value of $\rho_1$.
\end{remark}

\subsection{Motivation from physics literature: geometric boundary measures on both ends}\label{sec:motivation}
Our original motivation came from %
some  results in the mathematical physics literature, where generalized %
measures
for both end points arise in a random Motzkin path
representation for the stationary measure of a totaly asymmetric simple exclusion process (TASEP) with open boundaries
 in \cite{derrida04asymmetric} and for a more general open asymmetric simple exclusion process %
 \cite{barraquand2022Motzkin}.
As suggested by one of the reviewers, our results imply a similar representation for the stationary measures of a half-open totally asymmetric simple exclusion process (TASEP) that %
appeared
 (in more generality) in \citet{liggett75ergodic}, which we now explain.

Recall that a TASEP with open boundaries is a continuous time finite-state Markov process that models a configuration of particles at sites $\{1,\dots,L\}$ which can hop to the adjacent empty spot on the right with rate 1. Particles can arrive at site 1 at rate $\alpha>0$ and can leave  site $L$ at rate $\beta>0$. The stationary distribution of this Markov process is referred to as the stationary measure of open TASEP.

\citet[Section 2.5]{derrida04asymmetric} %
 provided the following description of the stationary measure of open TASEP with boundary parameters $\alpha+\beta>1$  via random Motzkin paths with generalized geometric
boundary weights. Fix $L\ge 1$. Consider a random Motzkin path $\vv\gamma\topp L$ as in our setup of length $L$  selected using  boundary weights $\alpha_n = \rho_0^n, \beta_n = \rho_1^n, n\ge 0$ with
\[
\rho_0=\frac{1-\alpha}{\alpha}\in(0,1) \qmand \rho_1=\frac{1-\beta}{\beta},
\] and with constant wedge weights \eqref{eq:abc} with $\sigma=2$.
In our setup with general geometric weights, the assumption \eqref{eq:rho_1} now becomes
$\rho_0\rho_1<1$ which is exactly $\alpha+\beta>1$.

To construct the desired stationary measure, %
{paraphrasing \citet[Section 2.5]{derrida04asymmetric}}
 we first select a random path $\gamma\topp L\in \mathcal{M}\topp L$, and independently i.i.d.~Bernoulli random variables $\{\zeta_k\}_{k\ge 1}$ with parameter $1/2$, and next set
\equh\label{eq:tau L} %
\tau_k\topp L:=\begin{cases}
  1 & \mbox{ if } \gamma\topp L_{k}-\gamma\topp L_{k-1}=1 \\
  0 & \mbox{ if }\gamma\topp L_{k}-\gamma\topp L_{k-1}=-1\\
  \zeta_k & \mbox{ if }\gamma\topp L_{k}-\gamma\topp L_{k-1}=0, \\
\end{cases}
\eque
$ k = 1,\dots,L$.
The law of $(\tau_1\topp L,\dots,\tau_L\topp L)$ is the stationary measure of open TASEP with boundary parameters %
$\alpha<1$, $\beta<1$,
 $\alpha+\beta>1$.

Regarding the left boundary processes alone, our limit theorem states a phase transition, and we see a counterpart result in \citet{liggett75ergodic} that we now explain. For the sake of comparison, we re-state our limit theorem for the left boundary process first.
Let $G_p$ denote a geometric random variable with parameter $0<p<1$, i.e., $\Pr(G_p=n)=(1-p)p^{n}$, $n\ge 0$.
\begin{corollary}\label{coro:1}
Consider the random Motzkin paths with $\alpha_n = \rho_0^n, \beta_n = \rho_1^n, n\ge 0$, with $\rho_0\in(0,1)$ and $\rho_0\rho_1<1$.
Then with $\hat \rho:=\max\{1,\rho_1\}$, %
\[%
   \ccbb{\gamma_k\topp L}_{k\ge 0}\weakto \ccbb{Z_k}_{k\ge 0} \quad \mbox{ as $L\to\infty$,}
\]%
 where on the right-hand side, $\sccbb{Z_k}_{k\ge 0}$ is the Markov chain with transition probabilities $\{\mathsf Q_{n,m}\topp {\hat \rho}\}_{n,m\ge 0}$ in \eqref{Z-trans}  and
  $Z_0=G_{\rho_0\hat \rho}+\wt G_{\rho_0/\hat \rho}$ is the sum of two independent geometric random variables.
\end{corollary}
\begin{proof}%
If $\rho_1<1$, then $\hat \rho=1$ and the conclusion follows from Theorem \ref{Thm-uni}, with \eqref{ini-law} giving the negative binomial  law for $Z_0$. If $\rho_1\in[1,1/\rho_0)$, then $\hat\rho=\rho_1$ and the result follows from
Theorem \ref{T-rho}. In this case,
\[
\EE z^{Z_0} = \frac{\sum_{n=0}^\infty \rho_0^n z^n \pp{\rho_1^{n+1}-1/\rho_1^{n+1}}/(\rho_1-1/\rho_1)}{\sum_{n=0}^\infty \rho_0^n \pp{\rho_1^{n+1}-1/\rho_1^{n+1}}/(\rho_1-1/\rho_1)}
=\frac{(1-\rho_0\rho_1)(1-\rho_0/\rho_1)}{(1-z \rho_0\rho_1)(1-z \rho_0/\rho_1)},
\]
identifying the law as desired.
The expression \eqref{eq:Z_0 1} gives the so-called $q$-negative binomial law, see  \cite [Theorem 3.1 with parameters  $k=2$, $q=\rho_1^2$, $\theta=\rho_0/\rho_1$]{charalambides2016discrete}.
\end{proof}

Corollary \ref{coro:1}  together with  \citet[Theorem 3.10(a)]{liggett75ergodic} gives an explicit description of the (non-unique) stationary measures of an open TASEP on a half-line, recovering all the stationary measures that appear in \citet[Theorem 1.8(a)]{liggett75ergodic} for $q=0$. Namely,  with $\tau_k\topp L := 0$ for all $k>L$ (recall \eqref{eq:tau L}),  %
 \citet[Theorem 3.10(a)]{liggett75ergodic} implies that
\equh\label{eq:liggett1975}
\left\{\tau_k\topp L\right\}_{k\ge 1}\weakto \left\{\tau_k\right\}_{k\ge 1},
\eque
 where the law of  $\{\tau_k\}_{k\ge 1}$ is a stationary measure of the half-open TASEP with Liggett's parameters $\la=\alpha=1/(1+\rho_0)>1/2$,  $q=0$, and the other parameter in \citet[Theorem 1.8(a)]{liggett75ergodic} given by $1-\beta=\rho_1/(1+\rho_1)$.
 Representation \eqref{eq:tau L} together with Corollary \ref{coro:1} give the following representation for Liggett's two-parameter stationary measure $\mu$: %
\begin{equation}
  \label{eq:tau infty}
  \tau_k:=\begin{cases}
  1 & \mbox{ if } Z_{k}-Z_{k-1}=1 \\
  0 & \mbox{ if }Z_{k}-Z_{k-1}=-1 \\
  \zeta_k & \mbox{ if }Z_{k}-Z_{k-1}=0,
\end{cases}
\end{equation}
$k\ge 1$,
and $\{\zeta_k\}_{k\ge 1}$ are i.i.d.~Bernoulli random variables with parameter $1/2$, independent from $Z$.  Obviously,  \eqref{eq:liggett1975} and Corollary \ref{coro:1} imply that the law of random variables
\eqref{eq:tau infty} is indeed a stationary measure of TASEP on $\NN$, which  was denoted by
$\mu(1/(1+\rho_0),\rho_1/(1+\rho_1))$
in \cite{liggett75ergodic}.
By Corollary \ref{coro:1}
the law of $\{\tau_k\}$  as defined by \eqref{eq:tau infty} does not depend on $\rho_1$ when $\rho_1\leq 1$ and this matches the phase transition in \citet[Theorem 1.8(a)]{liggett75ergodic}.

The limit process $\{Z_k\}_{k\ge 0}$, %
when $\hat\rho = 1$,
 can also be viewed as a discrete version of a Bessel process, with random initial position. %

Consider the limit
process
$Z$ in Corollary \ref{coro:1}.
By choosing the parameters $\rho_0,\rho_1$ appropriately, the increment process of $Z$ therein scales to the process corresponding to the %
non-Gaussian
 components of the conjectured {\em stationary measure of half-line KPZ fixed point} \citep{barraquand2022steady,Bryc-Kuznetsov-2021}, represented by the right-hand side of \eqref{eq:BD} below.
 In particular, in %
 \cite{Bryc-Wesolowski-2023} the following is shown.
With $\rho_0\topp n=1-u/\sqrt{n} $,
 $\rho_1\topp n=1-v/\sqrt{n}$, for fixed $u>0$,
 {$v\leq 0$,} $u+v>0$, and letting $\{Z_k\topp n\}_{k\ge 0}$ denote the process $Z$ with parameters $\rho_0\topp n,\rho_1\topp n$, we have %
 \equh\label{eq:BD}
 \frac{1}{\sqrt{n}}
 \ccbb{Z_{\floor{nt}}\topp n - Z_0\topp n}_{t\ge 0} \weakto \ccbb{2\pp{\sup_{0\le s\le t}B\topp v_s - \frac1{u+v}%
 \mathcal E}_+-B_t\topp v}_{t\ge 0},
 \eque
 as $n\to\infty$ in the space of $D([0,\infty))$,
 where on the right-hand side $B_t\topp v := \mathbb B_{2t/(2+\sigma)} + vt/(2+\sigma), t\ge 0$ with $\{\mathbb B_t\}_{t\ge 0}$ a standard Brownian motion, and $\mathcal E$ a standard exponential random variable independent from $\{B_t\topp v\}_{t\ge 0}$.

The paper is organized as follows. In Section \ref{sec:Viennot} we recall Viennot's formula. In Sections \ref{sec:3} and \ref{sec:4} we prove Theorems \ref{Thm-uni} and \ref{T-rho} respectively.

\section{Viennot's formula}\label{sec:Viennot}
A key ingredient of our proof is  a formula in \cite{Viennot-1984a}.
Given three sequences $\vv a, \vv  b, \vv c$ of the edge-weights for the Motzkin paths, following
\cite{Viennot-1984a} and \cite{flajolet09analytic} we define real polynomials  $p_{-1}(x)=0$, $p_0(x)=1,p_1(x),\dots$ by the three step recurrence
\begin{equation}
  \label{3-step}
  x p_n(x)=a_n p_{n+1}(x)+b_n p_n(x)+c_np_{n-1}(x),\quad n=0,1,2,\dots
\end{equation}
(With the usual conventions that $p_{-1}(x)=0$,  $p_0(x):=1$ and $a_n>0$, the recursion determines polynomials $\{p_n(x)\}$ uniquely, with $p_1(x)=(x-b_0)/a_0$.)
By Favard's theorem \citep{ismail09classical}, polynomials $\{p_n(x)\}$  are orthogonal with respect to a (possibly non-unique) probability measure $\nu$ on the real line.
It is well known that the $L_2$ norm   $ \|p_n\|_2^2:=\int_\RR (p_n(x))^2 \nu(dx)$ is given by the formula
\[%
  \|p_n\|_2^2=\prod_{k=1}^n \frac{c_k}{a_{k-1}}.
\]%
\begin{proposition}[Viennot's formula]\label{prop:Viennot}
We have
\equh
  \label{Viennot5}
  \int_\RR p_m(x)p_n(x) x^L\nu(dx) =\|p_n\|_2^2\sum_{\gamma\in \calM_{m,n}\topp L} w(\vv \gamma).
\eque
\end{proposition}
\begin{proof}
The identity \eqref{Viennot5} can be found in \citet[(5)]{Viennot-1984a}.
Since the weights in \cite{Viennot-1984a} are slightly less general   and the proof there is only sketched, we provide a self-contained proof for \eqref{Viennot5}.
We give a proof by induction on $L$. Recall $\mathfrak W\topp L_{m,n}$ in  \eqref{Wij}.
It is easy to see that for $L=1$ the only non-zero weights are $\mathfrak W\topp 1_{m,m-1}=c_m$, $\mathfrak W\topp 1_{m,m}=b_m$ and $\mathfrak W\topp 1_{m,m+1}=a_m$.  On the other hand, from the three step recursion \eqref{3-step} we see that the integrals $\int x p_m(x)p_{n}(x)\nu(dx)$ are zero, except for the following three cases:
$\int x p_m(x)p_{m-1}(x)\nu(dx)=c_m \|p_{m-1}\|_2^2$,  $\int x p_m(x)p_{m}(x)\nu(dx)=b_m  \|p_{m}\|_2^2$ and $\int x p_m(x)p_{m+1}(x)\nu(dx)=a_m \|p_{m+1}\|_2^2$. Thus $
M\topp 1_{m,n}=\mathfrak W\topp 1_{m,n}\|p_n\|_2^2$ for all $m,n\geq 0$, i.e. \eqref{Viennot5} holds for $L=1$.

Next, we use \eqref{w} to observe that
$$\mathfrak W_{m,n}\topp{L+1}=a_m \mathfrak W_{m+1,n}\topp L + b_m \mathfrak W_{m,n} \topp L+c_{m}\mathfrak W_{m-1,n} \topp L.$$
From the three step recursion \eqref{3-step} we see that the same recursion holds for $M_{m,n}\topp{L+1}$. This proves  \eqref{Viennot5} by mathematical induction.
\end{proof}

\section{Proof of Theorem \ref{Thm-uni}}\label{sec:3}

We first specify Viennot's formula to the setting of Theorem \ref{Thm-uni}. In this case, recursion \eqref{3-step} becomes
$$
x p_n(x)=p_{n+1}(x)+\sigma p_n(x)+p_{n-1}(x),
$$ so  polynomials $\{p_n\}$ are just the monic Chebyshev polynomials of the second kind with shifted argument, $p_n(x)=u_n(x-\sigma)$ where  $u_n(2\cos \theta)=\sin ((n+1) \theta)/\sin \theta$
 satisfy recursion
\[%
x u_n(x)=u_{n+1}(x)+u_{n-1}(x), \quad n=0,1,\dots
\]%
with $u_{-1}(x)\equiv 0,u_0(x)\equiv 1$. One readily checks from the above that
\begin{equation}\label{u(rho)}   u_n\pp{z+z\inv}=z^{-n}\sum_{k=0}^{n} z^{2k}= \begin{cases}     \displaystyle \frac{z^{n+1}-{z^{-(n+1)}}}{z-z\inv},  &\mbox{ if }z \ne \pm 1, \\  \\   (\pm1)^n ( n+1), &\mbox{ if } z=\pm 1.  \end{cases}\end{equation}
 In particular, polynomials $\{p_n\}_{n\ge 0}$ are orthogonal with respect to probability measure
\[
\nu(dx)=\frac{\sqrt{4-(x-\sigma)^2}}{2\pi}\inddd{|x-\sigma|<2}dx.
\]
Introduce the probability generating function
 for the end-points of a random Motzkin path
 of length $L$,
\[
\EE \left[z_0^{\gamma_0\topp L}z_1^{\gamma_L\topp L}\right]:= \sum_{\gamma\in \calM\topp L} z_0^{\gamma_0}z_1^{\gamma_L}\Pr_L(\gamma),\; L\ge 1,
\]
and also \begin{equation}
  \label{Phi-alpha}\Phi_{\vv \alpha}(z,x):=\sum_{n=0}^\infty \alpha_n z^n u_n(x),\quad \Phi_{\vv \beta}(z,x):=\sum_{n=0}^\infty \beta_n z^n u_n(x).
\end{equation}
\begin{lemma}
 Under the assumptions of Theorem \ref{Thm-uni}, we have
  \begin{equation}
    \label{Viennot5+}
    \frac{1}{2\pi}\int_{-2}^2 u_m(x)u_n(x)(x+\sigma)^L \sqrt{4-x^2}dx=\sum_{\vv\gamma\in \calM_{m,n}\topp L} w(\vv \gamma),
  \end{equation}
 and for  $z_0,z_1\in[0,1]$,
  \begin{equation}\label{Ends-g}
  \EE \left[z_0^{\gamma_0\topp L}z_1^{\gamma_L\topp L}\right]= \frac{\mathsf M_{\vvalpha,\vvbeta,L}( z_0, z_1)}{\mathsf M_{\vvalpha,\vvbeta,L}(1,1)},
  \end{equation}
  where
  \begin{equation}\label{ZL-g}
    \mathsf M_{\vvalpha,\vvbeta,L}(z_0,z_1):=\frac{1}{2\pi}\int_{-2}^{2} \Phi_{\vv \alpha}(z_0,x) \Phi_{\vv \beta}(z_1,x) (x+\sigma)^L\sqrt{4-x^2}d x.
  \end{equation}
\end{lemma}
\begin{proof}
It is clear that $\| p_n\|_2^2=1$.
 Then, \eqref{Viennot5+} follows from \eqref{Viennot5} in Proposition \ref{prop:Viennot} and a change of variables.

For \eqref{Ends-g}, first notice that, under the
assumptions %
\eqref{eq:integrable}
and $z\leq 1$,  the series
$\Phi_{\vv\alpha}(z,x),\Phi_{\vv\beta}(z,x)$
converge uniformly in $x\in[-2,2]$, and furthermore
\[
 \int_\RR \sum_{m,n=0}^\infty \alpha_n  {\beta_m} \abs{x^Lp_n(x)p_m(x)}\nu(dx)<\infty.
\]
(Recall also $|u_n(x)|\leq u_n(2)= n+1$ for $|x|\leq 2$.)
Then, we can write
\[
  \EE \left[z_0^{\gamma_0\topp L}z_1^{\gamma_L\topp L}\right] = \frac{\sif m0\sif n0\alpha_mz_0^m\beta_nz_1^n\sum_{\vv\gamma\in\calM_{m,n}\topp L}w(\vv\gamma)}{\sif m0\sif n0\alpha_m\beta_n\sum_{\vv\gamma\in\calM_{m,n}\topp L}w(\vv\gamma)}.
\]
So it suffices to show
\[
\mathsf M_{\vvalpha,\vvbeta,L}(z_0,z_1) = \sif m0\sif n0\alpha_mz_0^m\beta_nz_1^n\sum_{\vv\gamma\in\calM_{m,n}\topp L}w(\vv\gamma).
\]
To see this, it suffices to start from the right-hand side above and apply \eqref{Viennot5+} and Fubini's theorem.
\end{proof}
We also need the following.
\begin{lemma}\label{L-lim-F}  If $\sigma>0$ and $F$ is a continuous function on $[-2,2]$   then
   \begin{equation}\label{Lim-F}
  \lim_{L\to\infty}\frac{\int_{-2}^2F(x)(x+\sigma)^L \sqrt{4-x^2}dx}{\int_{-2}^2 (x+\sigma)^L\sqrt{4-x^2}dx}=F(2).
\end{equation}
\end{lemma}
This is Lemma \ref{Lim-F-mu} applied to the semicircle law $\mu$ with $R=2$.
\begin{proof}[Proof of Theorem \ref{Thm-uni}]
Recall that weak convergence of discrete-time processes means convergence of finite-dimensional distributions \citep{billingsley99convergence}. For integer-valued random variables, the latter
follows from convergence of probability generating functions. We will therefore show
 that for every fixed $K$, $z_0,z_1\in(0,1]$ and $t_1,\dots,t_K, s_1,\dots,s_K>0$
\equh\label{FixK}
\lim_{L\to\infty}  \EE \bb{z_0^{\gamma_0\topp L} \prod_{j=1}^K t_j^{\gamma_j\topp L-\gamma_{j-1}\topp L} \prod_{j=1}^K s_j^{\gamma_{L-j}\topp L-\gamma_{L+1-j}\topp L}z_1^{\gamma_L\topp L}}
=\EE\bb{ z_0^{X_0} \prod_{j=1}^K t_j^{X_j-X_{j-1}}}\EE\bb{ z_1^{X_0'} \prod_{j=1}^K s_j^{X_j'-X_{j-1}'}}.
\eque
Indeed, the above expressions determine uniquely the corresponding probability generating functions for small enough arguments. For example,
$$\EE\bb{ \prod_{j=0}^K v_j^{X_j}}=\EE\bb{ z_0^{X_0} \prod_{j=1}^K t_j^{X_j-X_{j-1}}}$$
with
$z_0=v_0\dots v_K$ and $t_j=v_jv_{j+1}\dots v_K$.

We introduce a tri-diagonal matrix
\[%
\mM_t:=\left[\begin{matrix}
  \sigma & t & 0 &0 &\cdots \\
  1/t & \sigma &t &0 &\\
  0 & 1/t & \sigma&t &\\
 0 & 0& 1/t & \ddots &\ddots\\
 \vdots &&&\ddots&\ddots
\end{matrix}\right],
\]%
and column vectors
\[%
  \vec V_\vvalpha(z) :=
\left[\begin{matrix}
  \alpha_0\\
    \alpha_1z\\
\vdots \\
 \alpha_n z^n \\ \vdots
\end{matrix} \right],\quad
\vec W_\vvbeta(z) := \left[\begin{matrix}
  \beta_0 \\ \beta_1 z   \\ \vdots \\ \beta_n z^n \\ \vdots
\end{matrix} \right],\quad \vec U(x):= \left[\begin{matrix}
  u_0(x) \\ u_1(x)    \\ \vdots \\ u_n(x) \\ \vdots
\end{matrix} \right],
\]%
where $\{u_k(x)\}_{k\ge 0}$ are the monic Chebyshev polynomials of the second kind that already appeared in \eqref{Phi-alpha}.
The key identity of the proof is the following:
\begin{multline}\label{*}
\EE \bb{z_0^{\gamma_0\topp L} \prod_{j=1}^K t_j^{\gamma_j\topp L-\gamma_{j-1}\topp L} \prod_{j=1}^K s_j^{\gamma_{L-j}\topp L-\gamma_{L+1-j}\topp L}z_1^{\gamma_L\topp L}}
\\=\frac{1}{\mathfrak C_{\vvalpha,\vvbeta,L}} \Vec V_\vvalpha(z_0)^T \mM_{t_1}\mM_{t_2}
\cdots \mM_{t_K} \mM_1^{L-2K} \mM_{1/s_K}\cdots \mM_{1/s_2}\mM_{1/s_1}\vec W_\vvbeta(z_1).
\end{multline}
Here $\mathfrak C_{\vvalpha,\vvbeta,L}=\mathsf M_{\vvalpha,\vvbeta,L}(1,1)$ is given by \eqref{ZL-g} with $z_0=z_1=1$ and
\begin{equation}\label{M2L-K}
   \mM_1^{L-2K} =\frac{1}{2\pi} \int_{-2}^2 \vec U(x) \vec U(x)^T(x+\sigma)^{L-2K} \sqrt{4-x^2}dx.
\end{equation}
 Indeed, the $(m,n)$-entry of $\mM_1^{L-2K}$ is $\mathfrak M_{m,n}\topp {L-2K} = \sum_{\vv\gamma\in\calM_{m,n}\topp {L-2K}}w(\vv \gamma)$, to which we applied Viennot's formula \eqref{Viennot5+}.

Note that on the right-hand side of \eqref{*}, the terms depending on $L$ are %
\[
\frac{\mM_1^{L-2K}}{\mathfrak C_{\vvalpha,\vvbeta,L}}  = \frac{\mM_1^{L-2K}}{\mathsf M_{\vvalpha,\vvbeta,L}(1,1)}
 =  \frac{\frac1{2\pi}\int_{-2}^2 (x+\sigma)^{L-2K}\sqrt{4-x^2}dx}{\mathsf M_{\vvalpha,\vvbeta,L}(1,1)}\cdot \frac{\mM_1^{L-2K}}{\frac1{2\pi}\int_{-2}^2 (x+\sigma)^{L-2K}\sqrt{4-x^2}dx},
\]
and for each ratio we apply Lemma \ref{L-lim-F}.
Namely, for the first ratio we have
 \begin{align*}\frac{\mathsf M_{\vvalpha,\vvbeta,L}(1,1)}{\frac{1}{2\pi}\int_{-2}^2(x+\sigma)^{L-2K}\sqrt{4-x^2} dx}&=
 \frac{\frac{1}{2\pi}\int_{-2}^2(x+\sigma)^{L}\Phi_{\vv \alpha}(1,x)\Phi_{\vv \beta}(1,x)\sqrt{4-x^2} dx}{\frac{1}{2\pi}\int_{-2}^2(x+\sigma)^{L-2K}\sqrt{4-x^2} dx}\\
 & \to (2+\sigma)^{2K}\Phi_{\vv\alpha}(1,2)\Phi_{\vv\beta}(1,2)%
 = C_{\vv\alpha}C_{\vv\beta}(2+\sigma)^{2K},
 \end{align*}
 as $L\to\infty$. %
For the second ratio, formally applying Lemma \ref{L-lim-F} entry-wise we get %
\equh
     \label{***}
        \lim_{L\to \infty} \frac{\mM_1^{L-2K}}{{\frac{1}{2\pi}\int_{-2}^2(x+\sigma)^{L-2K}\sqrt{4-x^2} dx}}  = \vec U(2) \times \vec U(2)^T
        = \left[\begin{matrix}
    1\\2\\3\\\vdots
  \end{matrix}\right]\times \left[\begin{matrix}
    1& 2& 3&\cdots
  \end{matrix}\right].
\eque
Putting this into \eqref{*}, and using the identities $\vec a^T \vec b=\vec b^T\vec a$ and $\mM_{1/s}^T=\mM_{s}$
 we rewrite the second inner product that arises on the right hand side of \eqref{*} as
\[
\left[\begin{matrix}
    1& 2& 3&\cdots
  \end{matrix}\right] \mM_{1/s_K}\cdots \mM_{1/s_2}\mM_{1/s_1}\vec W(z_1)
    = \vec W(z_1)^T \mM_{s_1}\cdots \mM_{s_K} \left[\begin{matrix}
    1\\2\\3\\\vdots
  \end{matrix}\right].
\]
We arrive at
\begin{multline}\label{**}
  \lim_{L\to\infty}  \EE \bb{z_0^{\gamma_0\topp L} \prod_{j=1}^K t_j^{\gamma_j\topp L-\gamma_{j-1}\topp L} \prod_{j=1}^K s_j^{\gamma_{L-j}\topp L-\gamma_{L+1-j}\topp L}z_1^{\gamma_L\topp L}}
  \\
  = \pp{ \frac{1}{C_{\vv\alpha} (2+\sigma)^{K}} \Vec V_\vvalpha(z_0)^T \mM_{t_1}\mM_{t_2}\cdots \mM_{t_K} \left[\begin{matrix}
    1\\2\\3\\\vdots
  \end{matrix}\right]}
  \times \pp{ \frac{1}{C_{\vv\beta} (2+\sigma)^{K}} \vec W_\vvbeta(z_1)^T \mM_{s_1}\cdots \mM_{s_K} \left[\begin{matrix}
    1\\2\\3\\\vdots
  \end{matrix}\right]}.
\end{multline}
    The rigorous argument that avoids formal entry-wise limit \eqref{***} is based on expanding the entire expression on the right hand of \eqref{*} using the integral representation \eqref{M2L-K} for the middle matrix $ \mM_1^{L-2K}$. To apply  Lemma \ref{L-lim-F} we rewrite this expression  so that all series and sums appear under the  integral. This  use of Fubini's theorem  is justified by  convergence of the series $\sum_n \alpha_n u_n(x)$ and
    $\sum_n \beta_n u_n(x)$ which is uniform over $x\in[-2,2]$. After an application of Lemma \ref{L-lim-F}, all appearances of $u_n(x)$ in the series get replaced by $u_n(2)=n+1$.  The same the expression arises by expanding into a series the matrix product on the right hand side of \eqref{**}. We omit cumbersome details of this re-write.

In view of this decomposition of the limit \eqref{**} into two factors of the same form, we need only to identify the first factor. Write $\vec X_{j:K} = (X_j,\dots,X_K)$ for simplicity.
 To this end we note that  for any integrable function $F$ we have %
\begin{align*}
   \EE &\bb{t^{X_j-X_{j-1}}F(\vec X_{j:K})\middle|X_{j-1}=n}
   = \sum_{\epsilon\in\{0,\pm1\}}t^\epsilon \EE\bb{F(\vec X_{j:K})\middle| X_j = n+\epsilon}\Pr(X_j=n+\epsilon| X_{j-1} = n)\\
     & = t\cdot\frac{n+2}{(2+\sigma)(n+1)}\EE\bb{F(\vec X_{j:K})\middle|X_{j}=n+1}
     +1\cdot
   \frac{\sigma}{2+\sigma} \EE\bb{F(\vec X_{j:K})\middle|X_{j}=n}
  \\  &\quad  +\frac{1}{t}\cdot \frac{n}{(2+\sigma)(n+1)}\EE\bb{F(\vec X_{j:K})\middle|X_{j}=n-1}.
\end{align*}
In matrix notation, the above is the same as %
\begin{equation}\label{eq:1}
   \frac{1}{2+\sigma}\mM_{t} \left[\begin{matrix}[1.8]
    1\cdot\EE\bb{F(\vec X_{j:K})\middle|X_{j}=0}
    \\2 \cdot
    \EE\bb{F(\vec X_{j:K})\middle|X_{j}=1}
\\\vdots
  \end{matrix}\right]
  = \left[\begin{matrix}[1.8]
    1\cdot \EE\bb{t^{X_j-X_{j-1}} F(\vec X_{j:K})\middle|X_{j-1}=0}
    \\2\cdot
   \EE\bb{t^{X_j-X_{j-1}} F(\vec X_{j:K})\middle|X_{j-1}=1}
    \\\vdots
  \end{matrix}\right].
\end{equation}
Starting with $j=K$, $t=t_K$  and constant $F(X_K)=1$, we  get
\begin{multline*}
\frac{1}{C_{\vv\alpha}  (2+\sigma)^{K}} \Vec V_\vvalpha(z_0)^T \mM_{t_1}\mM_{t_2}\dots \mM_{t_K} \left[\begin{matrix}[1.8]
    1\\2
    \\\vdots
  \end{matrix}\right]
  \\=
  \frac{1}{C_{\vv\alpha}  (2+\sigma)^{K-1}} \Vec V_\vvalpha(z_0)^T \mM_{t_1}\mM_{t_2}\dots \mM_{t_{K-1}} \left[\begin{matrix}[1.8]
    1 \cdot \EE[t_K^{X_K-X_{K-1}} | X_{K-1}=0]\\
    2\cdot \EE[t_K^{X_K-X_{K-1}} | X_{K-1}=1]\\
    \vdots
  \end{matrix}\right].
  \end{multline*}
By applying iteratively \eqref{eq:1}, we arrive at
  \begin{align*}
 & \frac{1}{C_{\vv\alpha}  (2+\sigma)^{K}} \Vec V_\vvalpha(z_0)^T \mM_{t_1}\mM_{t_2}\dots \mM_{t_K} \left[\begin{matrix}[1.8]
    1\\2
    \\\vdots
  \end{matrix}\right]
  \\
  &=
    \frac{1}{C_{\vv\alpha} } \Vec V_\vvalpha(z_0)^T\left[\begin{matrix}[1.8]
    1\cdot \EE[t_1^{X_1-X_0}\dots
    t_{K-1}^{X_{K-1}-X_{K-2}} t_K^{X_K-X_{K-1}}|X_{0}=0]
    \\  2\cdot \EE[t_1^{X_1-X_0}\dots t_{K-1}^{X_{K-1}-X_{K-2}}t_K^{X_K-X_{K-1}}|X_{0}=1]
\\\vdots
  \end{matrix}\right]
  \\
  &=
  \frac{1}{C_{\vv\alpha}}\sum_{n=0}^\infty \alpha_n z_0^n (n+1)\EE\bb{\prodd r1K t_r^{X_r-X_{r-1}}
  \mmid X_{0}=n}
 = \EE\bb{ z_0^{X_0} \prod_{j=1}^K t_j^{X_j-X_{j-1}}},
\end{align*}
with
$\Pr(X_0 = n) = \frac{\alpha_n(n+1)}{C_{\vv\alpha}}$ as claimed.
Similarly, the second factor is %
$\EE\bb{ z_1^{X_0'} \prod_{j=1}^K s_j^{X_j'-X_{j-1}'}}$,
 proving \eqref{FixK}.
\end{proof}
\section{Proof of Theorem \ref{T-rho}}\label{sec:4}

As previously, we fix $K$ and use matrix representation \eqref{*} for the generating function. The analysis is a little more involved.
We let $\mM_t$ and $\vec U(x)$ be as before, but this time work with the
column vector
\[
\vec W_{\rho_1}(z) := \left[\begin{matrix}
  1 \\ \rho_1 z   \\ \vdots \\ (\rho_1 z)^n \\ \vdots
\end{matrix} \right],\quad %
\rho_1>0.
\]
Taking  $z_1=1$ in \eqref{*} we  write
\begin{multline}\label{*-rho}
\EE \bb{z_0^{\gamma_0\topp L} \prod_{j=1}^K t_j^{\gamma_j\topp L-\gamma_{j-1}\topp L}\prodd k1Ks_k^{\gamma_{L-k}\topp L-\gamma_{L+1-k}\topp L }}
\\=\frac{1}{\mathfrak C_{\vvalpha,\rho_1,L}} \vec V_\vvalpha(z_0)^T \mM_{t_1}\mM_{t_2}\cdots \mM_{t_K} \mM_1^{L-2K}\mM_{1/s_K}\cdots\mM_{1/s_1}\vec W_{\rho_1}(1),
\end{multline}
where this time the normalizing constant is $\mathfrak C_{\vvalpha,\rho_1,L}:=\vec V_\vvalpha(1)^T \mM_1^{L} \vec W_{\rho_1}(1)$.

We observe that
$$\mM_t \vec W_{\rho_1}(1)=\pp{\frac{1}{t\rho_1}+\sigma+\rho_1 t} \vec W_{\rho_1}(1) - \vec R_1, $$
where vector $\vec R_1=[
  \frac{1}{t\rho_1}, 0, 0, \dots
]^T$ has only one non-zero entry.  Using this recurrently, we see that
\[
\mM_{1/s_K}\cdots \mM_{1/s_1}\vec W_{\rho_1}(1) = \prodd j1K \pp{\frac{s_j}{\rho_1}+\sigma+\frac{\rho_1}{s_j}}\vec W_{\rho_1}(1) - \vec R_K,
\]
 where vector $\vec {R}_K=[ w_0, w_1, \dots, w_{K-1},0,\dots]^T$ has only $K$ non-zero entries $w_j$. These entries do not depend on $L$, and their exact expressions are irrelevant for the limit theorem as shown a moment later. That is, writing
\equh\label{eq:r}
\EE \bb{z_0^{\gamma_0\topp L} \prod_{j=1}^K t_j^{\gamma_j\topp L-\gamma_{j-1}\topp L}\prodd k1Ks_k^{\gamma_{L-k}\topp L-\gamma_{L+1-k}\topp L }}
 = \Phi_L - r_L,
\eque
with
\begin{align*}
\Phi_L &:=\frac{1}{\mathfrak C_{\vvalpha,\rho_1,L}} \cdot \prodd j1K \pp{\frac{s_j}{\rho_1}+\sigma+\frac{\rho_1}{s_j}}\cdot \vec V_\vvalpha(z_0)^T \mM_{t_1}\mM_{t_2}\cdots \mM_{t_K} \mM_1^{L-2K}\vec W_{\rho_1}(1),\\
r_L&:=\frac{1}{\mathfrak C_{\vvalpha,\rho_1,L}}  \vec V_\vvalpha(z_0)^T \mM_{t_1}\mM_{t_2}\cdots \mM_{t_K} \mM_1^{L-2K}\vec R_K,
\end{align*}
we shall show that $\Phi_L$ has the desired limit as $L\to\infty$ and $\lim_{L\to\infty} r_L=0$.

The key step for %
$\rho_1\ge 1$
is to first extend the integral representation used in the proof of Theorem \ref{Thm-uni} in Lemma \ref{L-rho-I} below. To this end, for $\rho>0$, we introduce a probability measure
   measure $\mu_{\rho}$ of  mixed type given by
\[%
    \mu_\rho(dx):=\frac{1}{2\pi} \frac{\sqrt{4-x^2}}{1-x \rho +\rho^2 }\inddd{|x|<2}dx+\left(1-\frac{1}{\rho^2}\right)_+ \delta_{\rho+\tfrac1\rho}(dx), \quad\rho>0.
\]%
 (Here $x_+:=\max\{0,x\}$.) We also write
 \[
 \mu_0(dx) :=\frac{\sqrt{4-x^2}}{2\pi}\inddd{|x|<2}dx.
 \]
At the critical value %
$\rho=1$,
we have
\[
\mu_1(dx) = \frac1{2\pi}\sqrt{\frac{2+x}{2-x}}\inddd{|x|<2}dx.
\] \begin{lemma}\label{L-rho-I}  For $\rho_1>0, L>0$ we have
\equh\label{VMW2I}
 \mM_1^{L}\vec W_{\rho_1}(1)
     =
  \int_\RR(x+\sigma)^{L}\vec U(x) \mu_{\rho_1}(dx).
\eque
 \end{lemma}
 \begin{proof}
Recall that $\vec W_{\rho_1}(1) = [1,\rho_1,\rho_1^2,\rho_1^3,\dots]^T$. The equation is about two infinite-dimensional vectors. We examine the entries indexed by $m\in\Z_{\ge 0}$ (with the convention that the first entry is indexed by 0, the second by 1, etc).
Introduce
\[%
    h_m(\rho):=\sum_{n=0}^\infty\rho^n  \sum_{\vv \gamma \in \calM_{m,n}\topp {L} } w(\vv \gamma), \quad m\in\Z_{\ge 0}, \; \rho\in\CC.
\]%
Note that for every $m$ fixed, the summation over $n$ is a sum of %
a
finite number of non-zero terms (for those $n$ such that $|n-m|\le L$), and hence
$h_m$ is a polynomial in variable $\rho\in\CC$.
The $m$-indexed entry of the left-hand side of \eqref{VMW2I} is then
\[
\pp{\mM_1^{L}\vec W_{\rho_1}(1)}_m
    =\summ n0\infty\rho_1^n \sum_{\vv \gamma \in \calM_{m,n}\topp {L} } w(\vv \gamma) = h_m(\rho_1). \]
 Thus for $|\rho|<1$ from  \eqref{Viennot5+} we see that $h_m(\rho)=h_m\topp 0(\rho)$ where
 \begin{align}
   \label{f1**}
   h_m\topp 0(\rho)
&   := \sum_{n=0}^\infty \rho^n \frac{1}{2\pi}\int_{-2}^2 u_m(x)u_n(x)(x+\sigma)^{L} \sqrt{4-x^2}dx\\
&   = \int_{-2}^2(x+\sigma)^{L}u_m(x) \frac{\sqrt{4-x^2}}{2\pi(1-\rho x+\rho^2)}dx.\nonumber
 \end{align}
 Note that in the above we used the well-known identity for %
 the
 Chebyshev polynomials
 \equh\label{eq:identity}
 \sif n0 \rho^n u_n(x) = \frac1{1-\rho x+\rho^2}, \quad\mfa |\rho|< 1.
 \eque
  Our goal is to extend the integral representation in the last formula  of \eqref{f1**} to a larger domain by explicit analytic continuation.

Assume %
$|\rho|>0$ from now on. We first re-write \eqref{f1**} as a complex integral. Substituting $x=2\cos \theta$, and then $z=e^{i\theta}$, in \eqref{f1**}  we get
\begin{align}\label{f2}
h_m\topp 0(\rho)&= \frac{1}{2\pi}\int_{-2}^{2} u_m(x) \frac{ (\sigma+x)^{L}}{ 1-x \rho +\rho^2}\sqrt{4-x^2}d x
   \\
    &  =\frac{1}{2\pi}\int_{0}^{\pi} u_m(2\cos \theta)\frac{ 4\sin^2\theta(\sigma+2\cos \theta )^{L}}{ 1-2 \rho\cos \theta+\rho^2}  d\theta\nonumber
    \\
  & =\frac{1}{4\pi}\int_{-\pi}^{\pi} u_m(2\cos \theta)\frac{ 4\sin^2\theta(\sigma+2\cos \theta )^{L}}{ 1-2 \rho\cos \theta+\rho^2}  d\theta
\nonumber   \\& =
   \pp{-\frac{1}{2}}\cdot \frac{1}{2\pi i}\oint_{|z|=1} u_m\pp{z+\frac 1z}\frac{(z^2-1)^2\left(\sigma+z+\tfrac1z\right)^{L}}{
  (1-\rho z)(z-\rho)} \frac{dz}{z^2}, %
  \nonumber
\end{align}
$ |\rho|<1.$
Next, in the last line of \eqref{f2} one can replace the contour $|z|=1$ by $|z|=r$, and this replacement is valid as long as the circle does not cross any pole of the integrand, that is,  for $r\in(1,1/|\rho|)$.
Letting the contour cross the pole at $1/\rho$ and
 adding half of the residue at $z=1/\rho$ (because of the additional factor $-1/2$ in front of the integral),
we then arrive at
\[
h_m(\rho) = h_m\topp 0(\rho) = h_{m,r}\topp 1(\rho) \qmwith |\rho|<1, r>\frac1{|\rho|},
\]
with
\begin{multline}\label{f3}
h_{m,r}\topp 1(\rho)
:=   -\frac{1}{4\pi i}\oint_{|z|=r}u_m\pp{z+\frac 1z}\frac{(z^2-1)^2\left(\sigma+z+\tfrac1z\right)^{L}}{
   (1-\rho z)(z-\rho)} \frac{dz}{z^2} \\
   +\frac12u_m\pp{\rho+\frac1\rho}\frac{\rho^2-1}{\rho^2}   \left(\rho
   +\frac{1}{\rho}+\sigma \right)^{L}.
\end{multline}
By analytic extension, for $r$ fixed $h_{m,r}\topp 1(\rho)$ can be extended to  all $\rho$ such that $1/r<|\rho|<r$.
In particular,
\[
h_m(\rho)=h_{m,r}\topp 1(\rho), \quad  \mbox{ for all $\rho$ such that } \frac1r<|\rho|<r.
\]
   Next, consider the expression \eqref{f3} for $r>1$ %
   and $\rho$ such that $|\rho|\in(1,r)$, and deform the contour of integration back to $|z|=1$. This
subtracts
    half of the residue of the integrand at $z=\rho$.
So for
\begin{multline}
   \label{f4}
h_m\topp 2(\rho):=
-    \frac{1}{4\pi i}\oint_{|z|=1}u_m\pp{z+\frac 1z}\frac{(z^2-1)^2\left(\sigma+z+\tfrac1z\right)^{L}}{(1-\rho z)(z-\rho)} \frac{dz}{z^2}
 \\+u_m\pp{\rho+\frac 1\rho} \left(1-\frac1{\rho ^2}\right) \left(\rho
   +\frac{1}{\rho }+\sigma \right)^{L},
\end{multline}
we have
$h_{m,r}\topp 1(\rho) = h_m\topp2(\rho)$ for all $\rho\in\CC$ such that $|\rho|\in(1,r)$.
 Note that $r$ can be taken arbitrarily large.
   Therefore,
   \equh\label{f5}
   h_m(\rho)=h_m\topp 2(\rho) \quad \mbox{ for all $\rho$ such that } |\rho|>1.
   \eque

Returning back to the real arguments, we see that   \eqref{f2}, \eqref{f4} and \eqref{f5} can be combined together  into a single formula which for $\rho>0,\rho\ne 1$ gives
\[
h_m(\rho)=\frac{1}{2\pi}\int_{-2}^{2}u_m(x) \frac{ (\sigma+x)^{L}}{1-x \rho +\rho^2}\sqrt{4-x^2}d x
+
  \left(1-\frac{1}{\rho^2}\right)_+ u_m\pp{\rho+\frac1\rho}  \left(\rho
 +\frac{1}{\rho }+\sigma \right)^{L} .
\]
This formula extends to $\rho=1$ by continuity.
This proves \eqref{VMW2I}.
 \end{proof}
 We also need the following.

 \begin{lemma}\label{L-lim-F-rho} If $\sigma>0$, $\rho_1\geq 1$ and $F$ is a continuous function on $[-2,\rho_1+1/\rho_1]$   then
   \begin{equation}\label{Lim-F-rho}
  \lim_{L\to\infty}\frac{\int_\RR F(x)(x+\sigma)^L \mu_{\rho_1}(dx)}{\int_{\RR} (x+\sigma)^L\mu_{\rho_1}(dx)}=F\pp{\rho_1+\frac1{\rho_1}}.
\end{equation}

\end{lemma}
This is Lemma \ref{Lim-F-mu} applied to $\mu=\mu_{\rho_1}$ with $R=\rho_1+1/\rho_1$.

\begin{proof}[Proof of Theorem \ref{T-rho}]
We first show that  with $\rho_1\ge1$, the end-point of a random Motzkin path is not tight.
This can be easily seen from the generating function which,
using \eqref{*}, for $z_1\in(0,1/\rho_1)$ takes the form
$$
  \EE\bb{z_1^{\gamma_L\topp L}} =\frac{\vec V_\vvalpha(1)^T \mM_1^L\vec W_{\rho_1}(z_1)}{\vec V_\vvalpha(1)^T \mM_1^L\vec W_{\rho_1}(1)}
 =
\frac{  \int_\RR(x+\sigma)^L \vec V_\vvalpha(1)^T\vec U(x)\mu_{\rho_1z_1}(dx)}{\int_\RR(x+\sigma)^L \vec V_\vvalpha(1)^T\vec U(x)\mu_{\rho_1}(dx)},
$$
where we used \eqref{VMW2I} twice.
Since  $z_1\rho_1<1$, by \eqref{Lim-F} the numerator becomes
$$\int_\RR\frac{(x+\sigma)^L}{1-x z_1\rho_1 +z_1^2\rho_1^2} \vec V_\vvalpha(1)^T\vec U(x)\mu_{0}(dx)\sim
\frac{\vec V_\vvalpha(1)^T\vec U(2)}{(1-z_1\rho_1)^2}  \int_\RR (x+\sigma)^L\mu_{0}(dx).$$
Using \eqref{Lim-F-rho} in the denominator, we get
$$
\int_\RR(x+\sigma)^L \vec V_\vvalpha(1)^T\vec U(x)\mu_{\rho_1}(dx)\sim \vec V_\vvalpha(1)^T\vec U(\rho_1+1/\rho_1) \int_\RR(x+\sigma)^L \mu_{\rho_1}(dx).
$$
Therefore, %
\[
 \EE\bb{z_1^{\gamma_L\topp L}} \sim \frac{\vec V_\vvalpha(1)^T\vec U(2)}{(1-z_1\rho_1)^2\vec V_\vvalpha(1)^T\vec U(\rho_1+1/\rho_1)} \; \frac{\int_\RR(x+\sigma)^L \mu_{0}(dx)}{\int_\RR(x+\sigma)^L \mu_{\rho_1}(dx)}\to 0 \mbox{ as $L\to\infty$}.
\]
Indeed,    using \eqref{L-lim-F-mu} with $R=\rho_1+1/\rho_1$ and $F(x)=1-\rho_1x+\rho_1^2$ we have
\begin{equation}  \label{previously}
\lim_{L\to \infty} \frac{\int_\RR(x+\sigma)^L \mu_{0}(dx)}{\int_\RR(x+\sigma)^L \mu_{\rho_1}(dx)}
=\lim_{L\to \infty} \frac{\int_\RR F(x)(x+\sigma)^L \mu_{\rho_1}(dx)}{\int_\RR(x+\sigma)^L \mu_{\rho_1}(dx)}=F(\rho_1+1/\rho_1)=0.
\end{equation}

We next prove the joint convergence.
Applying Lemma \ref{L-rho-I} to \eqref{*-rho}  and recalling \eqref{eq:r}, we have
\begin{align}
\Phi_L&= \prodd j1K \pp{\frac{s_j}{\rho_1}+\sigma+\frac{\rho_1}{s_j}}\cdot \frac{ \vec V_\vvalpha(z_0)^T \mM_{t_1}\mM_{t_2}\cdots \mM_{t_K} \mM_1^{L-2K}\vec W_{\rho_1}(1)}{
\vec V_\vvalpha(1)^T   \mM_1^{L}\vec W_{\rho_1}(1)}\nonumber
\\
&= \prodd j1K \pp{\frac{s_j}{\rho_1}+\sigma+\frac{\rho_1}{s_j}}\cdot \frac{\vec V_\vvalpha(z_0)^T \mM_{t_1}\mM_{t_2}\cdots \mM_{t_K}  \int_\RR(x+\sigma)^{L-2K}\vec U(x) \mu_{\rho_1}(dx)}{
\int_\RR(x+\sigma)^{L-2K} \mu_{\rho_1}(dx)}\label{*-rho-1}
\\
& \quad\times \frac{
\int_\RR(x+\sigma)^{L-2K} \mu_{\rho_1}(dx)}{
\int_\RR(x+\sigma)^{2K}(x+\sigma)^{L-2K}\vec V_\vvalpha(1)^T \vec U(x) \mu_{\rho_1}(dx)}.\nonumber
\end{align}
Note that function
\[%
F(x):=\vec V_\vvalpha(z_0)^T \mM_{t_1}\mM_{t_2}\cdots \mM_{t_K}  \vec U(x)
\]
 is continuous on $[-2,\rho_1+1/\rho_1]$.
Indeed, since $|u_n(x)|\leq n+1$, continuity is obvious for $x\in[-2,2]$. Next, any $x\in[2,\rho_1+1/\rho_1]$ can be written as $x=\rho+1/\rho$ with unique $\rho\in[1,\rho_1]$,  and we have for some constant $C>0$,
\[%
 \left|F\pp{\rho+\frac1\rho}\right|\leq C\sif n0 \alpha_n \rho^{-n}\abs{\sum_{k=0}^n \rho^{2k}} \leq C \sif n0 (n+1) |\rho|^n\alpha_n +C \sif n0  \frac{(n+1)\alpha_n}{\rho^n}<\infty,\]
 where we needed the assumption \eqref{eq:rho_1}. It then follows that $F$ is also continuous on $[2,\rho_1+1/\rho_1]$.
Therefore applying Lemma \ref{L-lim-F-rho} to the two fractions on the right hand side of \eqref{*-rho-1}, we get
\equh\label{*rho}
 \lim_{L\to\infty} \Phi_L
 = \prodd j1K \frac{s_j/\rho_1+\sigma+\rho_1/s_j}{\rho_1+1/\rho_1+\sigma}\times \frac{ \vec V_\vvalpha(z_0)^T \mM_{t_1}\mM_{t_2}\cdots \mM_{t_K} \vec U(\rho_1+1/\rho_1)}{\mathfrak C_{\vvalpha,\rho_1} (\rho_1+1/\rho_1+\sigma)^{K}},
\eque
where
\[%
  \mathfrak C_{\vvalpha,\rho_1}:= \vec V_\vvalpha(1)^T\vec U(\rho_1+1/\rho_1) = \sif n0\alpha_nu_n\pp{\rho_1+\frac1{\rho_1}}.
\]
(In the last calculation, recall \eqref{u(rho)}.)

We first recognize
\equh\label{eq:s}
\prodd j1K \frac{s_j/\rho_1+\sigma+\rho_1/s_j}{\rho_1+1/\rho_1+\sigma} = \prodd j1K \EE s_j^{\xi_j}.
\eque
\sloppy For the second fraction on the right-hand side of \eqref{*rho}, as %
before
 we
 write $\vec Z_{j:K} = (Z_j\topp {\rho_1},\dots,Z_K\topp {\rho_1})$ for simplicity, and we drop $\rho_1$ from the notation. We first write the formulas for $\rho_1>1$.
 We note that  for any integrable function $F$ we have,
in matrix notation,
\begin{multline}
   \frac{1}{\rho_1+1/\rho_1+\sigma}\mM_{t} \left[\begin{matrix}[1.8]
    (\rho_1-1/\rho_1)\cdot\EE\bb{F(\vec Z_{j:K})\middle|Z_{j}=0}
    \\(\rho_1^2-1/\rho_1^2) \cdot
    \EE\bb{F(\vec Z_{j:K})\middle|Z_{j}=1}
\\\vdots\\
  \end{matrix}\right]
  \\
  = \left[\begin{matrix}[1.8]
   (\rho_1-1/\rho_1)\cdot \EE\bb{t^{Z_j-Z_{j-1}} F(\vec Z_{j:K})\middle|Z_{j-1}=0}
    \\(\rho_1^2-1/\rho_1^2)\cdot
   \EE\bb{t^{Z_j-Z_{j-1}} F(\vec Z_{j:K})\middle|Z_{j-1}=1}
    \\\vdots
  \end{matrix}\right].\label{eq:2}
\end{multline}

Recall that (see \eqref{u(rho)})   \[
   \vec U\pp{\rho_1+\frac1{\rho_1}} = \frac1{\rho_1-1/\rho_1}\bb{\begin{matrix}[1.8]
   \rho_1-1/\rho_1\\
   \rho_1^2-1/\rho_1^2\\
   \vdots
   \end{matrix}}.
   \]
   We  get
\begin{align*}
& \frac{\vec V_\vvalpha(z_0)^T \mM_{t_1}\cdots \mM_{t_K}\vec U(\rho_1+1/\rho_1)}{\mathfrak C_{\vvalpha,\rho_1}(\rho_1+1/\rho_1+\sigma)^K} \\
&=
 \frac{\Vec V_\vvalpha(z_0)^T \mM_{t_1}\mM_{t_2}\dots \mM_{t_K} }{\mathfrak C_{\vvalpha,\rho_1} (\rho_1-1/\rho_1) (\rho_1+1/\rho_1+\sigma)^{K}} \left[\begin{matrix}[1.8]
    \rho_1-1/\rho_1\\\rho_1^2-1/\rho_1^2\\
    \vdots
  \end{matrix}\right]
  \\& =
  \frac{\Vec V_\vvalpha(z_0)^T \mM_{t_1}\mM_{t_2}\dots \mM_{t_{K-1}}}{\mathfrak C_{\vvalpha,\rho_1} (\rho_1-1/\rho_1)  (\rho_1+1/\rho_1+\sigma)^{K-1}}  \left[\begin{matrix}[1.8]
    (\rho_1-1/\rho_1) \cdot \EE[t_K^{Z_K-Z_{K-1}} | Z_{K-1}=0]\\
    (\rho_1^2-1/\rho_1^2)\cdot \EE[t_K^{Z_K-Z_{K-1}} | Z_{K-1}=1]
    \\
    \vdots
  \end{matrix}\right].
\end{align*}
Eventually, we have
\begin{align}
&\frac1{\mathfrak C_{\vvalpha,\rho_1}(\rho_1+1/\rho_1+\sigma)^K}\vec V_\vvalpha(z_0)^T \mM_{t_1}\cdots \mM_{t_K}\vec U(\rho_1+1/\rho_1) \label{eq:t}\\
&=   \frac{1}{\mathfrak C_{\vvalpha,\rho_1} }\sum_{n=0}^\infty \alpha_n z_0^n \frac{\rho_1^{n+1}-1/\rho_1^{n+1}}{\rho_1-1/\rho_1}\EE\bb{
\prodd r1K t_r^{Z_r-Z_{r-1}}
\mmid Z_{0}=n}\nonumber
 \\
&   = \EE\bb{ z_0^{Z_0} \prod_{j=1}^K t_j^{Z_j-Z_{j-1}}},\nonumber
\end{align}
with
\equh\label{eq:Z_0}
\proba(Z_0 = n) = \frac{\alpha_n(\rho_1^{n+1}-1/\rho_1^{n+1})}{\mathfrak C_{\vvalpha,\rho_1}(\rho_1-1/\rho_1)}, \quad n\in\ZZ_{\ge 0}.
\eque
For $\rho_1 = 1$, replacing $(\rho_1^{n+1}-1/\rho_1^{n+1})/(\rho_1-1/\rho_1)$ by $n+1$, we have the same derivation, \eqref{eq:t} and \eqref{eq:Z_0} correspondingly. In fact, \eqref{eq:2} becomes \eqref{eq:1} (with $X$ replaced by $Z$).

Combining \eqref{*rho}, \eqref{eq:s} and \eqref{eq:t}, we have proved that
\[
\lim_{L\to\infty}\Phi_L = \left(\prodd j1K\EE s_j^{\xi_j}\right)\EE\bb{ z_0^{Z_0} \prod_{j=1}^K t_j^{Z_j-Z_{j-1}}}.
\]
 In view of \eqref{eq:r}, it remains to show that
$\lim_{L\to\infty}r_L =0$. First, recall that
\[
r_L=\frac{1}{\mathfrak C_{\vvalpha,\rho_1,L}}  \vec V_{\vvalpha}(z_0)^T \mM_{t_1}\mM_{t_2}\cdots \mM_{t_K} \mM_1^{L-2K}\vec R_K.
\]
\arxiv{  Denote by $\mI_1$  the   forward  shift with ones   above   the main diagonal and let $\mI_{-1}=\mI_1^T$ be the    backward shift. Let
$$\mI_k:=\begin{cases}
  (\mI_1)^k & k>0 \\
   (\mI_{-1})^{-k} & k<0 \\
   \mI & k=0
\end{cases}$$
be the corresponding matrices of shifts by $k\in\ZZ$. Since $\mI_k\mI_j- \mI_{k+j}$ can be non-zero only in the main $|k|\vee |j|\times |k|\vee |j|$  block, and $\mM_t=\sigma \mI_0 + t \mI_1+ 1/t \mI_{-1}$, we see that
 $\mM_{t_1}\mM_{t_2}\cdots \mM_{t_K}$ is a linear combination of the shift  matrices
$\mI_j$, $j=-K,\dots,K$   and a block matrix $\begin{bmatrix}
  \mB & \vv 0 \\
  \vv 0 & \vv 0
\end{bmatrix}$ where $\mB$ is a   $K\times K$ matrix.
That is,
$$\vec V_{\vvalpha}(z_0)^T  \mM_{t_1}\mM_{t_2}\cdots \mM_{t_K}=\sum_{j=-K}^K c_j (\mI_{-j}\vec V_{\vvalpha}(z_0))^T+\begin{bmatrix}
 b_1& \dots & b_K
\end{bmatrix}.$$
Hence, the numerator of $r_L$ is
$$
 r_L\; \mathfrak C_{\vvalpha,\rho_1,L}= \sum_{j=-K}^K c_j (\mI_{-j}\vec V_{\vvalpha}(z_0))^T\mM_1^{L-2K}\vec {R}_K + \begin{bmatrix}
 b_1& \dots & b_K
\end{bmatrix} \mM_1^{L-2K}\vec {R}_K.
$$
We now show that this expression  is of order $\int_{\RR}(x+\sigma)^{L-2K}\mu_0(dx)$.

We first note that by   Lemma~\ref{L-lim-F}  the least term is
\begin{multline*}
\begin{bmatrix}
 b_1& \dots & b_K
\end{bmatrix} \mM_1^{L-2K}\vec {R}_K= \sum_{m=1}^K\sum_{n=0}^{K-1} b_m w_n \int_{\RR}(x+\sigma)^{L-2K}u_{m-1}(x)u_n(x)\mu_0(dx)
\\ \sim  \sum_{m=1}^K\sum_{n=0}^{K-1} b_m w_n u_{m-1}(2)u_n(2) \int_{\RR}(x+\sigma)^{L-2K}\mu_0(dx).
\end{multline*}
Next, we note that if $j\geq 0$ then
$$(\mI_{-j}\vec V_{\vvalpha}(z_0))^T =[0, \dots,0, \alpha_0, \alpha_1z_0,\dots] $$
and if $j<0$ then
$$(\mI_{-j}\vec V_{\vvalpha}(z_0))^T =[\alpha_{-j} z_0^{-j}, \alpha_{-j+1} z_0^{-j+1}, \dots]  $$
}
Setting $\alpha_k:=0$ for $k<0$ we therefore obtain
 \begin{multline*}
    (\mI_{-j}\vec V_{\vvalpha}(z_0))^T \mM_1^{L-2K}\vec {R}_K=
 \sif m0  \alpha_{m-j}z_0^{m-j}\summ n0{K-1} w_{n}
\int_\RR (x+\sigma)^{L-2K}u_m(x)u_n(x) \mu_0(dx)\\
= \summ n0{K-1} w_n\int_\RR(x+\sigma)^{L-2K}\pp{\sif m0\alpha_{m-j}z_0^{m-j}u_m(x)}u_n(x)\mu_0(dx)
\\
\sim \pp{\sif m0(m+1)\alpha_{m-j} z_0^{m-j}}\summ n0{K-1} w_n  u_n(2)\cdot \int_\RR (x+\sigma)^{L-2K}\mu_0(dx),
 \end{multline*}
 as $L\to\infty$,
where the last step follows from Lemma~\ref{L-lim-F}.
Thus, $ r_L \mathfrak C_{\vvalpha,\rho_1,L}$ is   of  order $\int_{\RR}(x+\sigma)^{L-2K}\mu_0(dx)$ as $L\to\infty$.
On the other hand, the normalizing constant is
\begin{align*}
\mathfrak C_{\vvalpha,\rho_1,L}=\vec V_\vvalpha(1)^T \mM_1^{L}\vec W_{\rho_1}(1)& = \int_\RR (x+\sigma)^L\vec V_\vvalpha(1)^T\vec U(x)\mu_{\rho_1}(dx) \\&= \int_\RR (x+\sigma)^{L}\sif n0\alpha_nu_n(x)\mu_{\rho_1}(dx)\\
& \sim  \sif n0 \alpha_nu_n\pp{\rho_1 + \frac1{\rho_1}}\cdot \int_\RR(x+\sigma)^L\mu_{\rho_1}(dx),
\end{align*}
as $L\to\infty$,
where the last step follows from Lemma~\ref{L-lim-F-rho}.
\arxiv{Lemma~\ref{L-lim-F-rho} also implies that
$$\lim_{L\to\infty} \frac{\mathfrak C_{\vvalpha,\rho_1,L-K}}{\mathfrak C_{\vvalpha,\rho_1,L}}= (\rho_1+1/\rho_1+\sigma)^{-K}$$
}

 From  \eqref{previously}    we obtain  $\int_\RR(x+\sigma)^{L-2K}\mu_0(dx)= o(\int_\RR(x+\sigma)^L\mu_{\rho_1}(dx))$ as $L\to\infty$. %
  Thus $r_L\to 0$ and  the proof is completed.
\end{proof}
\begin{remark}\label{rem:geometric}
The choice of geometric boundary measures is crucial at two steps. The first is that it leads to the generating function for Chebyshev polynomials in \eqref{eq:identity}, which is key for Lemma \ref{L-rho-I}. Second, it is key for the derivation of \eqref{eq:t}.
\end{remark}

\appendix
\section{A technical lemma}
We need the following version of \cite[Lemma 1]{Bryc-Wesolowski-2015-asep}.
 \begin{lemma}\label{L-lim-F-mu} Suppose $\mu$ is a probability measure such that $\rm{supp} (\mu)\subset [-2,R]$ for some $R\geq 2$ and that $R\in\rm{supp} (\mu)$. If $\sigma>0$,   and $F$ is a continuous function on $[-2,R]$   then $\int_{\RR} (x+\sigma)^L\mu(dx)$ is non-zero for large $L$ and
   \begin{equation}\label{Lim-F-mu}
  \lim_{L\to\infty}\frac{\int_\RR F(x)(x+\sigma)^L \mu(dx)}{\int_{\RR} (x+\sigma)^L\mu(dx)}=F(R).
\end{equation}

\end{lemma}

\begin{proof}
Subtracting $F(R)$ from both sides, we see that it suffices to prove \eqref{Lim-F-mu} with $F(R)=0$.

  Fix $\eps>0$. By continuity of $F$ at $x=R$, there is $\delta>0$ with $\delta\leq \sigma/2$ and $\delta\leq 2$ such that for all $ x\in[R-\delta,R]$ we have
$|F(x)|<\eps$.

We rewrite the   fraction under the limit \eqref{Lim-F-rho} as follows: %
\begin{align*}
  \frac{\int_{-2}^RF(x)(x+\sigma)^L \mu(dx)}{\int_{-2}^R (x+\sigma)^L\mu(dx)}
 &  =\frac{\int_{-2}^{R-\delta}F(x)(x+\sigma)^L \mu(dx)+\int_{R-\delta}^RF(x)(x+\sigma)^L \mu(dx)}{\int_{0}^R (x+\sigma)^L\mu(dx)+\int_{-2}^0 (x+\sigma)^L\mu(dx)}
\\& =
\frac{\int_{-2}^{R-\delta}F(x)(x+\sigma)^L \mu(dx) }{\int_{0}^R (x+\sigma)^L\mu(dx)\pp{1+\frac{\int_{-2}^0 (x+\sigma)^L\mu(dx)}{\int_{0}^R (x+\sigma)^L\mu(dx)}}}
\\
& \quad+\frac{ \int_{R-\delta}^R F(x)(x+\sigma)^L \mu(dx)}{\int_{0}^R (x+\sigma)^L\mu(dx)\pp{1+\frac{\int_{-2}^0 (x+\sigma)^L\mu(dx)}{\int_{0}^R (x+\sigma)^L\mu(dx)}}}.
\end{align*}
Since $|F(x)|$ is bounded on $[-2,R]$ by its supremum $\|F\|_\infty$,
$$\left|\int_{-2}^{R-\delta}F(x)(x+\sigma)^L \mu(dx)\right|\leq \|F\|_\infty \int_{-2}^{R-\delta}|x+\sigma|^L \mu(dx).$$
By continuity of $F$,  $$\left|\int_{R-\delta}^RF(x)(x+\sigma)^L \mu(dx)\right| \leq \eps \int_{R-\delta}^R(x+\sigma)^L \mu(dx) \leq \eps \int_{0}^R(x+\sigma)^L \mu(dx).$$
 To end the proof it suffices to show that
\begin{equation}
  \label{lims-mu}
\frac{\int_{-2}^{R-\delta}|x+\sigma|^L \mu(dx)}{\int_{0}^R (x+\sigma)^L\mu(dx)}\to 0\quad \mbox { and }\quad
\frac{\int_{-2}^0 |x+\sigma|^L\mu(dx)}{\int_{0}^R (x+\sigma)^L\mu(dx)}\to 0 .
\end{equation}

To prove \eqref{lims-mu}, we note that for $x\in[-2,R-\delta]$ we have
$\sigma-2\leq x+\sigma\leq R+\sigma-\delta$ so
$|x+\sigma|^L\leq \pp{\max\{2-\sigma,R+\sigma-\delta\}}^L\leq (R+\sigma-\delta)^L$ (where in the last bound we used $R\geq 2$ and $\delta\leq \sigma/2$.) and similarly for $-2\leq x<0$
we have $|x+\sigma|^L\leq \pp{\max\{2-\sigma,\sigma\}}^L$. Thus the numerators in \eqref{lims-mu} are bounded from above by $ (R+\sigma-\delta)^L$ and $\pp{\max\{2-\sigma,\sigma\}}^L$ respectively.

Next, we tackle the denominator in \eqref{lims-mu}. Since $\lim_{L\to\infty}\EE\bb{|X|^L}^{1/L}=\|X\|_\infty$, we see that $\pp{\int_{0}^R (x+\sigma)^L\mu(dx)}^{1/L}$ is arbitrarily close to  $R+\sigma$.
Thus, for any $\wt \delta>0$ with $\wt \delta<R+\sigma$ and all $L$ large enough we have
$$\int_{0}^R (x+\sigma)^L\mu(dx)\geq (R+\sigma-\wt \delta)^L.$$
For the first limit in \eqref{lims-mu}, we choose $\wt \delta=\delta/2$, so that the expression is bounded by
$(R+\sigma-\delta)^L/(R+\sigma-\wt \delta)^L\to 0$.
 For the second limit, we choose
$\wt \delta=\min \{\sigma,1\}$
so that the expression is bounded by $\pp{\max\{2-\sigma,\sigma\}}^L/(R+\sigma-\wt \delta)^L\leq
\pp{\max\{2-\sigma,\sigma\}}^L/(2+\sigma-\wt \delta)^L
= \max\{\frac{2-\sigma}{2}, \frac{\sigma}{1+\sigma}\}^L\to 0$.
This proves \eqref{lims-mu},  ending the proof.
\end{proof}

\subsection*{Acknowledgments}
We thank
Jacek Weso{\l}owski for formula \eqref{u(rho)}
and references.
We are grateful to an anonymous referee for %
helpful suggestions.

 The first author was supported in part by Simons Foundation/SFARI Award Number: 703475, US.
 The second author was supported in part by Army Research Office, US (W911NF-20-1-0139).
 Both authors acknowledge support from the  Charles Phelps Taft Research Center at the University of Cincinnati.

\def\polhk#1{\setbox0=\hbox{#1}{\ooalign{\hidewidth
  \lower1.5ex\hbox{`}\hidewidth\crcr\unhbox0}}}

\end{document}